\documentclass[a4paper, 11pt]{article}
\usepackage[margin=3cm]{geometry}

\usepackage[hang,small]{caption}
\usepackage[latin1]{inputenc}
\usepackage{amsmath}
\usepackage{amsfonts}
\usepackage{amssymb}
\usepackage{amsthm}
\usepackage{fancyhdr}
\usepackage{float}
\usepackage{cite}
\usepackage{enumerate}          
\usepackage{graphicx}           
\usepackage{fancybox}           
\usepackage{makeidx}
\usepackage{stmaryrd}				
\usepackage{url}

\numberwithin{equation}{section}

\theoremstyle{definition}

\newtheorem{defn}[equation]{Definition}

\newtheorem{rmk}[equation]{Remark}

\theoremstyle{plain}
\newtheorem{thm}[equation]{Theorem}
\newtheorem{prop}[equation]{Proposition}
\newtheorem{lem}[equation]{Lemma}
\newtheorem{cor}[equation]{Corollary}

\newcommand{\A}{\ensuremath{\mathcal{A}}}

\newcommand{\B}{\ensuremath{\mathcal{B}}}
\renewcommand{\c}{\ensuremath{\mathcal{C}}}
\newcommand{\C}{\ensuremath{\mathbb{C}}}
\renewcommand{\d}{\mathrm d}

\newcommand{\FID}{$\boxplus$-infinitely divisible}

\renewcommand{\H}{\ensuremath{\mathcal{H}}}

\newcommand{\hphi}{\ensuremath{\widehat{\phi}}}

\renewcommand{\l}{\ensuremath{\mathcal{L}}}
\newcommand{\map}{\ensuremath{\longrightarrow}}

\newcommand{\m}{\ensuremath{\mathfrak{M}}}

\newcommand{\N}{\ensuremath{\mathbb{N}}}
\newcommand{\NCP}{non-commutative probability space}
\newcommand{\one}{\ensuremath{\mathbf{1}}}

\newcommand{\p}{\ensuremath{\mathcal{P}}}
\newcommand{\Proc}{\ensuremath{\mathfrak{P}}}
\newcommand{\Q}{\ensuremath{\mathcal{Q}}}
\newcommand{\R}{\ensuremath{\mathbb{R}}}

\newcommand{\tensor}{\ensuremath{\otimes}}

\newcommand{\ul}[1]{\underline{#1}}
\newcommand{\unit}{\ensuremath{\mathbf{1}}}
\newcommand{\wh}[1]{\widehat{#1}}

\newcommand{\abs}[1]{\ensuremath{\left|#1\right|}}
\renewcommand{\Im}{\mathfrak{Im}}
\newcommand{\inv}[1]{\ensuremath{\frac{1}{#1}}}

\newcommand{\ts}[1]{\textsc{#1}}
\newcommand{\VNA}{von Neumann algebra}

\DeclareMathOperator*{\Leb}{L}

\DeclareMathOperator*{\NC}{NC}

\DeclareMathOperator*{\spt}{supp}

\title{Functionals of the Free Brownian Bridge}
\date{ }
\author{Janosch Ortmann\\ Warwick Mathematics Institute}

\makeindex

\begin{document}
\maketitle

\setlength\headheight{15pt}

\fancyhead[R]{\nouppercase{\rightmark}}
\fancyhead[L]{}
\cfoot{\thepage}

\bibliographystyle{acm}

\bibstyle{plain}

\abstract{\noindent We discuss the distributions of three functionals of the free Brownian bridge: its $\Leb^2$-norm, the second component of its signature and its L\'evy area. All of these are freely infinitely divisible. We introduce two representations of the free Brownian bridge as series of free semicircular random variables are used, analogous to the Fourier representations of the classical Brownian bridge due to \ts{L\'evy} and \ts{Kac}.}


\section{Introduction}

In this note we discuss the distributions of three non-commutative random variables defined in terms of a free Brownian bridge.

In his paper \cite{Levy50}, \ts{L\'evy} introduces the following representation of the Brownian bridge. Let $\xi_n, \eta_n$ be independent standard Gaussian random variables then the process defined by
\begin{align}
	\label{Eq:IntroLevy}
	\beta_{2\pi}(t) &= \sum_{n=1}^\infty \frac{\cos(nt)-1}{n\sqrt{\pi}} \xi_n+ \sum_{n=1}^\infty \frac{\sin(nt)}{n\sqrt{\pi}} \eta_n
	\intertext{defines a Brownian bridge on $[0,2\pi]$. Another representation is given by \ts{Kac} \cite{Kac50}. Retaining the notation for the $\eta_n$ it is a consequence of Mercer's theorem that the Gaussian process defined by}
	\label{Eq:IntroKac}
	\beta_1(t) & = \sum_{n=1}^\infty \frac{\sin(nt)}{n\pi}\, \eta_n
\end{align}
has the covariance kernel of a Brownian bridge. The analogue of the Gaussian distribution and processes in non-commutative probability theory are the semicircle law and semicircular processes. It turns out that the crucial properties of the Gaussian distribution needed for the observations above are shared by the semicircular law. Therefore if we replace $\xi_n,\eta_n$ by free standard semicirculars then (\ref{Eq:IntroLevy}) and (\ref{Eq:IntroKac}) define free Brownian bridges on $[0,2\pi]$ and $[0,1]$ respectively. We will use this fact to prove various properties of the square norm, the second component of the signature and the L\'evy area of the free Brownian bridge.

The $\Leb^2$-norm of the classical Brownian bridge was first considered by \ts{Kac} who used his representation (\ref{Eq:IntroKac}) to compute its Fourier transform. Further calculations were performed using \ts{Kac}'s work, see \ts{Tolmatz} \cite{Tolmatz_Bridge} and the references therein. We will compute the R-transform of the free analogue of this object and use the fact that its law is freely infinitely divisible to prove that it has a smooth density for which we give an implicit equation.

In \cite{CDM_Levy} \ts{Capitaine} and \ts{Donati-Martin} construct the second component $Z$ of the signature of the free Brownian motion. This process plays a role in the theory of rough paths, see \cite{CDM_Levy}, \ts{Lyons}\cite{Lyons98} and \ts{Victoir}\cite{Victoir04} for details. The second component of the signature is a process taking values in the tensor product of the underlying non-commutative probability space with itself. Equipped with the product expectation this is a probability space in its own right and we compute the R-transform of $Z$.  A connection between the cumulants of $Z$ and the number of \textit{2-irreducible meanders}, a combinatorial object introduced by \ts{Lando--Zwonkin}\cite{LandoZwonkin93} and further analysed by \ts{Di Francesco--Golinelli--Guitter} \cite{DGG_Meander} is pointed out.

Finally we apply the L\' evy-type representation to compute the R-transform of the L\' evy area corresponding to the free Brownian bridge. This random variable is also freely infinitely divisible. Once again this allows us to deduce that the law in question has a smooth density. Again we obtain an implicit equation. 

From the considerations involving free infinite divisibility it also follows that the support of the law of both L\'evy area and square norm is a single interval, in the former case symmetric about the origin, in the latter strictly contained in the positive half-line. In \cite{LDP_NCP} a large deviations principle is established for the blocks of a uniformly random non-crossing partition. This result allows us to determine the maximum of the support from the free cumulants. We obtain implicit equations that determine the essential suprema of L\'evy area and square norm.

\paragraph{Acknowledgements.} The author would like to thank his PhD advisor, Neil O'Connell for his advice and support in the preparation of this paper. We also thank Philippe Biane for helpful discussions and suggestions.


\section{Free Probability Theory}
\label{Sec:FPT}

We recall here some definitions and properties from free probability theory. For an introduction to the subject see for example \cite{VDN, VoicSF, HiaiPetz}.

\subsection{Freeness, Distributions, and Transforms}

Throughout let $(\A,\phi)$ be a \textit{non-commutative probability space}, i.e.  a unital \VNA\ equipped with a state $\phi$ on \A. We think of elements $a\in\A$ as non-commutative random variables and consider $\phi(a)$ to be the expectation of $a\in\A$. We will only consider self-adjoint $a\in\A$. Then there exists a compactly supported measure $\mu_a$ on \R, called the \textit{distribution} of $a$, such that
\begin{align*}
	\phi(a^n) &= \int t^n\,\mu_a(\d t)\quad\quad\forall\, n\in\N.
	\intertext{Recall that the \textit{Cauchy transform} of $\mu_a$ is defined to be}
	G_{\mu_a} (z)  = \int_\R\frac{\mu_a(\d t)}{t-z} &= \sum_{n=0}^\infty \phi(a^n) z^{-n-1}.
\end{align*}
\noindent Since $\mu_a$ is compactly supported the first equality defines an analytic map $G_{\mu_a}\colon\C^+\map\C^-$. The power series expansion is valid on a neighbourhood $U_a$ of infinity. We will also write $G_a$ for $G_{\mu_a}$. 

\begin{defn}
	Von Neumann subalgebras $\B_1,\ldots,\B_N$ of $\A$ are said to be \textit{free} if for every set of indices $\{r_j\}_{j=1}^m\subseteq\{1,\ldots,N\}$ and collection $\{a_j\in\B_{r_j}\colon 1\leq j\leq m\}$ such that $r_j\ne r_{j+1}$ and $\phi(a_j)=0\ \forall\, j\ $ we already have
	\begin{align*}
			\phi(a_1,\ldots,a_m)=0.
	\end{align*}
Random variables $a_1,\ldots,a_N$ are said to be free if the unital von Neumann algebras generated by the $a_j$ are free. 
\end{defn}

\par\noindent If $a$ and $b$ are free then the distribution of $a+b$ is uniquely determined by those of $a$ and $b$ (see Remark~\ref{Rmk:PropRTf}(2) below). Denote the laws of $a,b$ by $\mu_1,\mu_2$ respectively. Then the \textit{free convolution} of $\mu_1$ and $\mu_2$ is defined to be the distribution of $a+b$. Because self-adjoint elements of \A\ are determined by their distribution this induces a binary operation on the space of compactly supported probability measures, denoted $\boxplus$.\\

A partition $\pi$ of the set $\ul{n}=\{1,\ldots,n\}$ is said to be \emph{crossing} if there exist distinct blocks $V_1$, $V_2$ of $\pi$ and $x_j,y_j\in V_j$ such that $x_1<x_2<y_1<y_2$. Otherwise $\pi$ is said to be \emph{non-crossing}. Equivalently, arrange the numbers $1,\ldots,n$ clockwise on a circle and connect any two elements of the same block of $\pi$ by a straight line. Then $\pi$ is non-crossing if and only if the lines drawn are pairwise disjoint. Let $\NC(n)$ denote the set of non-crossing partitions on $\ul{n}$.

\begin{figure}[H]
	\begin{center}
	\includegraphics[scale = .2]{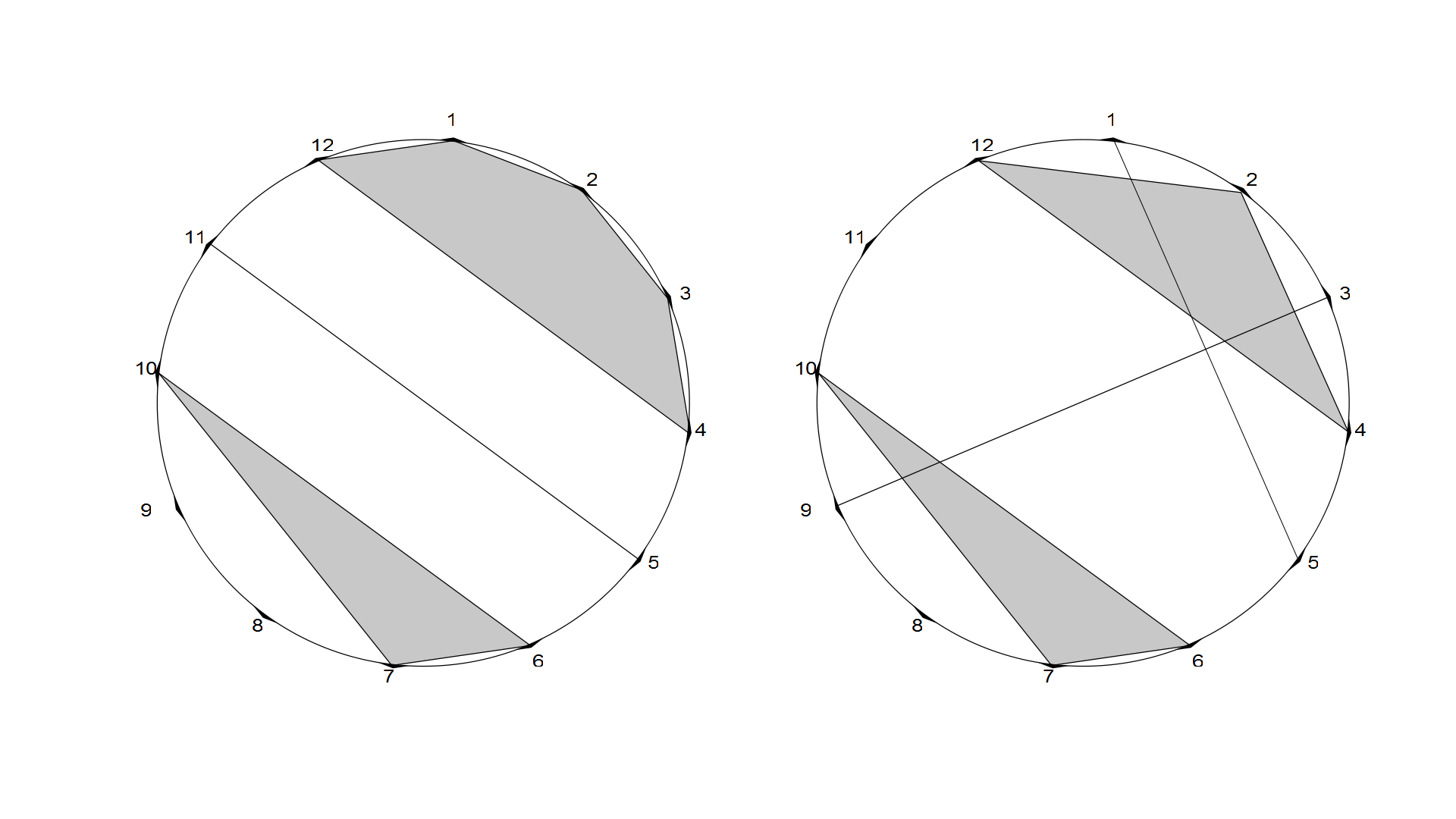}
	\hspace{5cm}
	\parbox{11cm}{\caption{\footnotesize{The partition $\{\{8\}, \{9\}, \{10,7,6\}, \{11,5\}, \{12,4,3,2,1\} \}$ is non-crossing, $\{\{5,1\}, \{8\}, \{9,3\},\{10,7,6\}, \{12,4,2\} \}$ is crossing.}}}
\end{center}

\end{figure}

\begin{defn}
    \label{Def:NCC}
	The \textit{free cumulants} of \A\ are defined to be the maps $k_n\colon\A^n\map\C$ ($n\in\N$) defined indirectly by the following system of equations:
    \begin{align}
			\label{Eq:DefCum}
        \phi(a_1,\ldots,a_n) &=  \sum_{\pi\in\text{NC}(n)} k_\pi[a_1,\ldots,a_n]
        \intertext{where $k_\pi$ denotes the product of cumulants according to the block structure of $\pi$. That is, if $V_1,\ldots,V_r$ are the components of $\pi\in\NC(n)$ then}
			\notag
        k_\pi[a_1,\ldots,a_n] & = k_{V_1}[a_1,\ldots,a_n]\ldots k_{V_n}[a_1,\ldots,a_n]
    \end{align}
    where, for $V=(v_1,\ldots,v_r)$ we just have $k_V[a_1,\ldots,a_n]=k_{\abs{V}}[a_{v_1},\ldots,a_{v_r}]$. 
\end{defn}

\par\noindent Note that (\ref{Eq:IntroKac}) has the form $\phi(a_1,\ldots,a_n)=k_n[a_1,\ldots,a_n]+$ lower order terms, so that we can find the $k_n$ inductively. Alternatively, (\ref{Eq:DefCum}) defines the $k_n$ by M\"obius inversion. See \cite{NicaSpeicher} for details.

\par We will write $k_n(a)$ for $k_n[a,\ldots,a]$. The \textit{R-transform} of a random variable $a\in\A$ is defined to be the formal power series
	\begin{align}
		\label{Eq:DefRT}
		R_a(z) & = \sum_{n=0}^\infty k_{n+1}(a) z^n.
		\intertext{If the law of $a$ has compact support then equation (\ref{Eq:DefRT}) defines an analytic function on a neighbourhood of zero \cite[Theorem 3.2.1]{HiaiPetz}. Moreover the Cauchy transform $G_a$ of $a$ is locally invertible on a neighbourhood of infinity and the inverse $K_a$ satisfies}
	\notag
	K_a(z) & = R_a(z) + \inv{z}.
\end{align}

\begin{rmk} \label{Rmk:PropRTf} The following three properties of the R-transform are easy to check using the continuity of $\phi$ and multilinearity of the cumulants.
\begin{enumerate}
\item \label{En:RCont} If $a_n\longrightarrow a\in\A$ then $R_{a_n}(z) \longrightarrow R_a(z)$ as $n\to \infty$
\item \label{En:Radd} If $a,b\in\A$ are free then $R_{a+b}(z) =R_a(z)+R_b(z)$
\item \label{En:RScale} For $\lambda\in\C$ we have $R_{\lambda a}(z) = \lambda R_a(\lambda z)$.
\end{enumerate}
\end{rmk}

\subsection{Semicircular Processes}

\label{Sec:SemiProc}

\begin{defn}
	A collection $\mathcal{S}=(s_j)_{j\in I}$ of non-commutative variables on \A\ is said to be a \textit{semicular family} with \textit{covariance} $\left(c(i,j)\right)_{i,j\in I}$ if the cumulants are given by
	\begin{align*}
		k_\pi[s_{j_1},\ldots s_{j_n}] & = \prod_{p\sim_\pi q} c(j_p,j_q).
	\end{align*}
\end{defn}

\noindent If $\mathcal{S}$ consists of a singleton $s_1$ and $r=2\sqrt{c(1,1)}$ then the distribution of $s_1$ is the \textit{centred semicircle law of radius} $r$, that is the measure $\sigma_r$ on \R\ given by

\begin{align*}
	\sigma_r(\d t)& = \frac{2}{\pi r^2}\sqrt{r^2-t^2}\ \mathbf{1}_{[-r,r]}(t)\, \d t.
\end{align*}

\noindent In particular $\sigma_2$ is also called the \textit{standard semicircle law} and non-commutative random variables with law $\sigma_r$ ($\sigma_2$) are referred to as \textit{(standard) semicirculars}. 
\par The semicircle law plays a similar role to the Gaussian distribution on classical probability theory. In particular there exists a central limit theorem \cite[Theorem 3.5.1]{VDN}, and a collection of random variables with a joint semicircular law is determined by its covariance. To be more precise we recall the following

\begin{prop}[\ts{Nica--Speicher}\cite{NicaSpeicher}, Proposition 8.19]
	Let $(s_i)_{i\in I}$ be a semicircular family of covariance $(c(i,j))_{i,j\in I}$ and suppose $I$ is partitioned by $I_1,\ldots, I_d$. Then the following are equivalent:
\begin{enumerate}
\item The collections $\{s_j\colon j\in I_1\},\ldots,\{s_j\colon j\in I_d\}$ are free
\item We have $c(r,j)=0$ whenever $r\in I_p$ and $j\in I_q$ with $p\ne q$.
\end{enumerate}
\end{prop}

\noindent In particular $\{s_j\colon j\in I\}$ is a free family if and only if $C=(c(r,j))_{r,j\in I}$ is diagonal.

\begin{defn}
	A process $(X(t))_{t\geq 0}$ on $\A$ is said to be a \textit{semicircular process} if for every $t_1,\ldots, t_n\in [0,\infty)$, the set $\left(X(t_1),\dots,X(t_n)\right)$ is a semircular family.
\end{defn}

\par\noindent By the considerations above the finite-dimensional distributions of a semicircular process are determined by the \textit{covariance structure} of the process, i.e. by the function $C\colon [0,\infty)^2\map\C$ defined by
\begin{align*}
	C(s,t) & = \phi(X(s) X(t)).
\end{align*}.

\subsection{The L\' evy Representation of the Free Brownian Bridge}

\label{Sec:FBM}

\begin{defn}
	A centred semicircular process $\left(\beta_T(t)\right)_{t\in [0,T]}$ on $\A$ is said to be a \textit{free Brownian bridge} on $[0,T]$ if its covariance structure is given by
	\begin{align*}
		\phi(\beta_T(s)\beta_T(t)) = s\wedge t - \frac{st}{T}.
	\end{align*}
\end{defn}

\begin{rmk}
	\label{Rmk:BridgeMotion}
	In analogy with classical probability it can be easily checked that if $\beta$ is a free Brownian bridge on $[0,1]$ and $\xi_0$ is a free standard semicircular free from $\{\beta(t)\colon t\in [0,1]\}$, then $X(t)=\xi_0 t+\beta(t)$ defines a \textit{free Brownian motion}, that is

\begin{enumerate}[(i)]
\item the distribution of $X(t)$ is a centred semicircular law with radius $t$;
\item $X(t)-X(s)$ is free from $\{X(r)\colon r\leq s\}$
\item $X(t)-X(s)$ has the same distribution as $X(t-s)$.
\end{enumerate}

\end{rmk}

\noindent The following proposition is the analogue of \ts{L\' evy}'s representation of the classical Brownian bridge \cite{Levy50}. Its proof follows from the fact that centred semicircular processes are determined by their covariance and that (non-commutative) covariances of the $\xi_n, \eta_n$ are the same as the (commutative) covariances of a corresponding independent family of standard Gaussian variables.

\begin{prop}
	\label{Thm:Levy}
	Let $\{\xi_n,\eta_m\colon (n,m)\in\N_0\times\N\}$ be a set of free standard semicircular variables in \A. Then the process $\beta_{2\pi}$ defined by
\begin{align}
	\label{Eq:Levy}
	\beta_{2\pi}(t)	& = \sum_{n=1}^\infty \frac{\cos(nt)-1}{n\sqrt{\pi}} \xi_n+ \sum_{n=1}^\infty \frac{\sin(nt)}{n\sqrt{\pi}} \eta_n
\end{align}
\noindent is a free Brownian bridge on $[0,2\pi]$.
\end{prop}

\subsection{A Representation for Centred Semicircular Processes}

\label{Sec:SeriesRepn}

In this section we show how \ts{Kac}'s representation\cite{Kac50} for the classical Brownian bridge on the unit interval can be translated into the setting of free probability. His method extends to all centred semicircular (or indeed Gaussian) processes, as follows. Everything relies on the following classical result from functional analysis, see \ts{Bollobas} \cite{Bollobas}.

\begin{thm}[Mercer's theorem]
	Let $K\colon [a,b]\times [a,b]\map \R$ be a non-negative definite symmetric kernel. Denote by $\H$ the Hilbert space $\Leb^2[a,b]$ and let $T_K$ be the operator on \H\ associated to $K$, that is,
	\begin{align}
		T_K (f)(s) & = \int_a^b K(s,t)\,f(t)\ \d t.
		\intertext{Then there exists an orthonormal basis $(f_n)_{n\in\N}$ of \H\ consisting of eigenfunctions of $T_K$ such that the corresponding eigenvalues $\lambda_n$ are non-negative, $f_n\in\c[a,b]$ whenever $\lambda_n\ne 0$ and}
		\label{Eq:Kernelsum}
		K(s,t) & = \sum_{n=1}^\infty \lambda_n f_n(s) f_n(t)
	\end{align}
	where the convergence is absolute and uniform, and hence also in $\Leb^2[a,b]$.
\end{thm}

\noindent We can use Mercer's theorem to represent any centred semicircular process as a series of free standard semicircular random variables, noting that if $Y$ is a centred semicircular process on $[a,b]$ then its covariance function $K$ defined  by $K(s,t)=\phi(Y(s)Y(t))$ is a non-negative symmetric kernel on $[a,b]$.

\begin{cor}
	\label{Thm:KacRep}
	Let $K, \H, (\lambda_n,f_n)_{n\in\N}$ be as in Mercer's theorem and let $(\eta_n)_{n\in\N}$ be a sequence of free standard semicirculars. Then the process $Y$ defined by
	\begin{align}
		\label{Eq:SCrep}
		Y(t) &= \sum_{n=1}^\infty \sqrt{\lambda_n} f_n(t)\, \eta_n
	\end{align}
	\noindent is a centred semicircular process of covariance $K$.
	\begin{proof}
		It is immediate that $Y$ is a centred semircircular process. Its covariance kernel is given by
		\begin{align*}
			\phi(Y(s)Y(t)) & = \sum_{m,n=1}^\infty \sqrt{\lambda_m\lambda_n} f_m(s) f_n(t) \phi(\eta_m\eta_n)\\
				& = \sum_{n=1}^\infty \lambda_n f_n(s) f_n(t) = K(s,t)
		\end{align*}
		\noindent by Mercer's theorem.
	\end{proof}
\end{cor}

\noindent For the free Brownian bridge on $[0,1]$ we have $K(s,t)=s\wedge t-st$. Solving the corresponding eigenvalue-eigenvector equation we obtain \ts{Kac}'s representation in the free setting:

\begin{align}
	\label{Eq:KacRep}
	\beta_1(t) & = \sum_{n=1}^\infty \frac{\sin(nt)}{n \pi}	\, \eta_n.
\end{align}


\ \\ 
\section{Square Norm of the Free Brownian Bridge}
\label{Sec:SqNorm}

In this section we consider the square-norm of a free Brownian bridge $\beta$ on  interval. Recall that $\A$ is a \VNA\ so that we can consider $\beta$ as a map from $[0,1]$ into a Banach space which is easily seen to be continuous. We can therefore use Riemann integration to define

\begin{align*}
	\Gamma & = \int_0^1 \beta(t)^2\, \d t
\intertext{where $\beta$ is a free Brownian bridge on $[0,1]$.  In this section we discuss the distribution of the non-commutative random variable $\Gamma$, using the representation (\ref{Eq:KacRep}). \ts{Kac}\cite{Kac50} showed that the Laplace transform of the commutative analogue of $\Gamma$ is given by}
	\wh{f}(p)&= \left(\frac{\sqrt{2p}}{\sinh{\sqrt{2p}}}\right)^{(1/2)}.
\end{align*}
\par\noindent Other properties, in particular the density function $f$, were computed, most recently by \ts{Tolmatz}\cite{Tolmatz_Bridge}.
\par We give here the R-transform of $\Gamma$ and an expression for its moments involving a sum over non-crossing partitions. Further below we show that the distribution $\mu_\Gamma$ of $\Gamma$ is freely infinitely divisible. This gives us some analytic tools to show that there exist $a,b\in\R$ with $0<a<b<1$ such that the support of $\mu_\Gamma$ is $[a,b]$ and that $\mu_\Gamma$ has a smooth positive density on $[a,b]$. We give an implicit equation and a sketch for the density.

Finally we use a result from \cite{LDP_NCP} to characterise the maximum $b$ of the support of $\mu_\Gamma$.

\subsection{The R-transform} 

\begin{prop}
	\label{Thm:Rtf}
	The R-transform of $\Gamma$ is given by
	\begin{align}
		\label{Eq:RGamma}
		R_\Gamma(z) & = \frac{1-\sqrt{z}\cot(\sqrt{z})}{2 z}.
	\end{align}
	\begin{proof}
	By orthonormality of the functions $\sin(nt)$ we have
	\begin{align}
	\label{Eq:freesqnorm}
		\Gamma &= \inv{\pi^2}\,\sum_{n=1}^\infty \inv{n^2} \eta_n^2.
\end{align}
	\noindent The square of a standard semicircular random variable is a free Poisson element of unit rate and jump size (\ts{Nica--Speicher}\cite{NicaSpeicher}, Proposition 12.13).  So the Cauchy transform of $\eta_n$ is given by \cite{NicaSpeicher}
\begin{align*}
	G_n(z) = \sum_{n=0}^\infty c_m z^{-m-1} & = \inv{z}C\left(\inv{z}\right) = \inv{2}+\sqrt{\inv{4}-\inv{z}}.
	\intertext{The free cumulants of $\eta_n^2$ are all equal to 1 and the R-transform is given by}
	R_n(z)  & = \inv{1-z},\quad\quad \abs{z}<1.
\end{align*}
\noindent Using the properties of the R-transform mentioned in Remark~\ref{Rmk:PropRTf}) we obtain for $\abs{z}<\pi^2$
\begin{align*}
	R_\Gamma(z) & = \sum_{n=1}^\infty \inv{\pi^2 n^2} R_n\left(\frac{z}{\pi^2 n^2}\right) = \sum_{n=1}^\infty \inv{n^2\pi^2 - z} = \frac{1-\sqrt{z}\cot(\sqrt{z})}{2 z}
\end{align*}
as claimed.
	\end{proof}
\end{prop}

\noindent The free cumulants of $\Gamma$ are therefore given by

\begin{align*}
	k_m & = \frac{\zeta(2m)}{\pi^{2m}} = (-4)^m\,\frac{B_{2m}}{(2m)!}
	\intertext{where $B_n$ is the $n^\text{th}$ Bernoulli number and $\zeta$ the Riemann zeta function. With (\ref{Eq:DefCum}) we obtain a formula for the moments involving a sum over non-crossing partitions:}
	\phi\left(\Gamma^n\right) & =\ \inv{\pi^{2n}}\sum_{\sigma\in\NC(n)} \prod_{r=1}^{m_\sigma} \zeta(2l_r^\sigma)	 = (-4)^n \sum_{\sigma\in\NC(n)}  \prod_{r=1}^{m_\sigma} \frac{B_{2l_r^\sigma}}{(2l_r^\sigma)!} 
\end{align*}
\noindent where $m_\pi$ denotes the number of equivalence classes of $\pi$ and $l_r^\pi$ is the size of the $r^\text{th}$ equivalence class of $\pi$.

While there does not seem to exist a closed-form expression for the inverse of $K_\Gamma(z)=R_\Gamma(z)-\inv{z}$ (and hence, by the Stieltjes inversion formula, for the density) we will describe some properties of the law $\mu_\Gamma$ of $\Gamma$. We will prove that $\mu_\Gamma$ is freely infinitely divisible, has a positive analytic density on a single interval and give an equation for the right end point of that interval.

\subsection{Free Infinite Divisibility}

\label{Sec:FIDBP}

The concept of infinite divisibility has a natural analogue in free probability theory. Noting that the square norm of the free Brownian bridge is freely infinitely divisible we will use the approach of P.~\ts{Biane} in his appendix to the paper \cite{BercoviciPata99} to prove that the law of $\Gamma$ has a smooth density on its support and give an implicit formula for that density.

\begin{defn}
	A compactly supported probability measure $\mu$ is said to be \textit{freely infinitely divisible} (or \textit{\FID}) if for every $n\in\N$ there exists a compactly probability measure $\mu_n$ such that
\begin{align*}
	\mu & = \mu_n^{\boxplus n} = \underbrace{\mu_n\boxplus\ldots\boxplus\mu_n}_{n\text{ times}}
\end{align*}
where $\boxplus$ denotes free convolution (Section~\ref{Sec:FPT}).
\end{defn}

\par\noindent Since each $\xi_n$ has a free Poisson distribution and is therefore freely infinitely divisible it follows that $\Gamma$ is also \FID.

Recall that the Cauchy transform $G_\Gamma$ of $\Gamma$ is an analytic map from the upper half plane $\C^+$ into the lower half plane $\C^-$, which is locally invertible on a neighbourhood of infinity, and that its local inverse is given by the K-transform $K_\Gamma$ where
	\begin{align*}
		K_\Gamma(z) & = R_\Gamma(z)+\inv{z} = \frac{3-\sqrt{z}\cot\left(\sqrt{z}\right)}{2z}.
	\end{align*}	
 From Proposition 5.12 in \ts{Bercovici}--\ts{Voiculescu} \cite{BercoviciVoiculescu93} and the infinite divisibility of $\Gamma$ it is straightforward to deduce the following result.

\begin{lem}
	\label{Thm:ImageGGamma}
	The law $\mu_\Gamma$ of the square norm of the free Brownian bridge can have at most one atom. Moreover its Cauchy transform $G_{\Gamma}$ is an analytic injection from $\C^+$ whose image is the connected component $\Omega$ in $\C^-$ of
\begin{align*}
	\wh\Omega & = \{z\in\C^-\colon \Im \left(K_\Gamma(z)\right)>0 \}
\end{align*}
	that contains $iy$ for small values of $y$.
\end{lem}

\par\noindent It will be useful to characterise the boundary $\partial\Omega$.

\begin{lem}
	For every $t\in (\pi,2\pi)$ there exists unique $r(t)>0$ such that $\Im\left[\left(K_\Gamma(r(t) e^{it})\right)\right]=0$. Moreover
\begin{align}
	\label{Eq:DerNonZero}
	\left.\frac{\partial}{\partial z} \Im K_\Gamma(z)\right|_{z=r(t)e^{it}}&\ne 0\quad\quad\qquad\forall\,t\in(\pi,2\pi).
\end{align}
\begin{proof}
	Fix $t\in(\pi,2\pi)$. The imaginary part of $K_\Gamma$ can be written in polar co-ordinates by
	\begin{align*}
			h_t(r) := K_\Gamma\left(r\, e^{it}\right) & = -\frac{3\sin(t)}{2r} + \frac{\gamma\sinh(\sigma\sqrt{r})\cosh(\sigma\sqrt{r}) + \sigma\sin(\gamma\sqrt{r})\cos(\gamma\sqrt{r})}{2\sqrt{r}\left(\sin^2(\gamma\sqrt{r}) + \sinh^2(\sigma\sqrt{r})\right)}
	\intertext{where $\sigma=\sin(t/2)$ and $\gamma=\cos(t/2)$. Define $g_t(r)=2r\,h_t(r^2)$. Then}
		g_t(r) &= -\frac{6\sigma\gamma}{r} + \frac{\sigma\sin\left(2\gamma r\right)+\gamma\sinh\left(2\sigma r\right)}{2\left[\sin^2(\gamma\sqrt{r}) + \sinh^2(\sigma\sqrt{r})\right]}.
	\end{align*}
	It is a lengthy but simple calculation to prove the existence of unique $\rho(t)>0$ such that $g_t(\rho(t))=0$ and that $g_t'(\rho(t))<0$. The result follows.
\end{proof}
\end{lem}

\par\noindent Therefore $\wh\Omega$ is actually simply connected: it is given by the area enclosed by the real axis and the curve $\lambda=\{r_t e^{it}\colon t\in (\pi,2\pi)\}$. In particular $\Omega=\wh\Omega$ and $\partial\Omega$ is a continuous simple curve. So Carath\'eodory's theorem applies, wherefore the analytic bijection $G_\Gamma\colon\C^+\map\Omega$ extends to a homeomorphism (denoted $\wh{G}_\Gamma$) from $\C^+\cup\R\cup\{\infty\}$ to the closure $\overline\Omega$ of $\Omega$ in $\C\cup\{\infty\}$.
\par Since $\Omega$ is bounded, so is its closure, whence $\wh{G}_\Gamma$ is finite on $\C^+\cup\R\cup\{\infty\}$. The set of isolated points of the support of $\mu_\Gamma$ is exactly the set of $t\in\R$ such that $\wh{G}_\Gamma(t)=\infty$ so $\spt( \mu_\Gamma)$ must be an interval $[a,b]$. From the Stieltjes inversion formula (see for example \cite{HiaiPetz}, p.93) it now follows that if we put for $x\in[a,b]$
	\begin{align}
		\Phi(x) & = -\inv{\pi}\lim_{y\to 0} \Im\left(G_\Gamma(x+iy)\right) = -\inv{\pi}\, \Im\left(\wh{G}_\Gamma(x)\right)
	\end{align}
then $\mu_\Gamma$ has density $\Phi$ with respect to Lebesgue measure on $[a,b]$. Since $K_\Gamma$ is the inverse of $G_\Gamma$ and because of (\ref{Eq:DerNonZero}) the implicit function theorem applies and hence $\Phi$ is smooth on $[a,b]$. Moreover it follows that
	\begin{align*}	
	\spt\mu_\Gamma &=K_\Gamma\left(\partial\Omega\cap\C^-\right)  =\left[K_\Gamma\left(r_{\pi+}\right) \wedge K_\Gamma\left(r_{2\pi-}\right), K_\Gamma\left(r_{\pi+}\right)\vee K_\Gamma\left(r_{2\pi-}\right)\right].
	\end{align*}
	where $r_{\pi+}=\lim_{s\downarrow 0} r_{\pi+s}$ and $r_{2\pi-}= \lim_{s\downarrow 0} r_{2\pi-s}$.
\par The operator $\Gamma$ is positive and its norm is less than 1, so the support of $\mu_\Gamma$ must be contained in the unit interval. We summarise the results of this section.

\begin{prop}
	\label{Thm:SqNormSummarise}
	There exist $a,b\in\R$ such that $0\leq a<b\leq 1$ and a positive smooth function $\Phi\colon [a,b]\map\R$ such that
	\begin{align}
		\mu_\Gamma(\d t) & = \Phi(t)\one_{[a,b]}(t).
	\end{align}
\par\noindent The function $\Phi$ is given by $\Phi(x) = -\inv{\pi}r(t_x)\sin(\tau_x)$ where $\tau_x\in (\pi,2\pi)$ is the unique solution to $K_\Gamma\left(r(\tau_x)\,e^{i\tau_x}\right)=x$.
\end{prop}

\par\noindent Below is a sketch of the density function based on numerical computations.\\

\begin{figure}[H]
	\begin{center}
	\includegraphics[scale = .5]{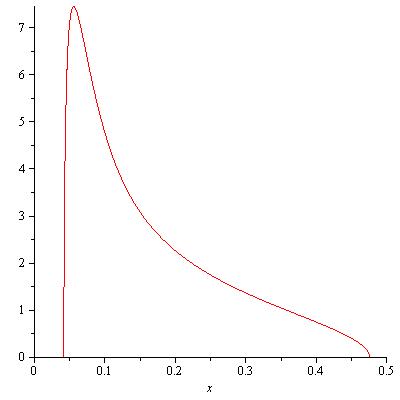}
	\parbox{11cm}{\caption{\footnotesize{Density of the $\Leb^2$-norm of the free Brownian bridge.}}}
\end{center}

\end{figure}

\subsection{The Maximum of the Support}

We now study the maximum of the support of $\mu_\Gamma$. We will need Theorem 5.4 from \cite{LDP_NCP}:

\begin{thm}
	\label{Thm:EdgePositive}
	Let $\mu$ be a compactly supported probability measure on $[0,\infty)$ such that its free cumulants $(k_j)_{j\in\N}$ are all positive. Then the right edge $\rho_\mu$ of the support of $\mu$ is given by
	\begin{align}
		\label{Eq:EdgePositive}
		\log\rho_\mu & = \sup\left\{\inv{m_1(p)}\sum_{m=1}^\infty p_m\log\left(\frac{k_m}{p_m}\right) + \frac{\Theta\big(m_1(p) \big)}{m_1(p)} \colon p\in\m_1^1(\N)\right\}
	\end{align}	
	where $\m_1^1(\N)=\{p\in\m_1(\N)\colon m_1(p)<\infty\}$ is the set of probability measures on \N\ with finite mean and $\Theta(m)=\log(m-1) - m\log\left(1-\inv{m}\right)$.
\end{thm}

\par\noindent It turns out that this variational problem can be solved using the method of Laplace multipliers. There exists a unique maximiser $p^*$ for the supremum on the right-hand side of (\ref{Eq:EdgePositive}). Using the series expansion of $\zeta(2n)$ and interchanging summation we obtain
\begin{align}
	\notag
	p^*_n & = \inv{m^*-1}\, \zeta(2n) \left(\frac{\gamma}{\pi}\right)^{2n}
	\intertext{where $\gamma$ is a rational function of $m^*$ and $m^*$ is the unique solution on $\left(\frac{2}{3},\infty\right)$ of the equation}
	\label{Eq:MStarSquare}
	m-3 & = \sqrt{4m^2-2m-6}\, \cot\left(\frac{4m^2-2m-6}{m-1} \right)
\end{align}
In the end we obtain an implicit equation for the right edge of the support of $\mu_\Gamma$:

\begin{prop}
	The number $b$ from Proposition~\ref{Thm:SqNormSummarise} is given by
	\begin{align*}
		b = \frac{\left(m^*\right)^2-m^*}{4\left(m^*\right)^2-2m^*-6}
	\end{align*}
	where $m^*$ is the unique solution of $(\ref{Eq:MStarSquare})$ on $\left(\frac{2}{3},\infty\right)$.
\end{prop}


\section{The Signature of the Free Brownian Bridge}

\label{Sec:SigFBB}

\subsection{Signature and Rough Paths}

In T.~\ts{Lyons}'s paper \cite{Lyons98} a new approach to differential equations driven by rough paths is proposed. For a general Banach-valued path $p\colon\R_+\map E$ we define, when this makes sense, the \textit{signature} of $p$ to be the process $S(p)$ taking values in the tensor algebra $T((E))=\bigoplus_{n=0}^\infty E^{\otimes n}$ whose $n^\text{th}$ component is given by the $n$-times iterated integral against $p$:
\begin{align*}
	S(p)_n(t) &= \int_{0<t_1<\ldots<t_n<t}\d p\,(t_1)\otimes\ldots\otimes\d p\,(t_n).
	\intertext{The signature is then used to solve general differential equations of the form}
	\d\, S(t) & = S(t)\tensor\d p\,(t).
\end{align*}

\par\noindent In order to show that this works if the path in question is a free Brownian motion $X$, \ts{Capitaine}--\ts{Donati-Martin}\cite{CDM_Levy} define an integral of a class of suitable processes $\Proc$ against $X$ that yields a process taking values in the tensor product $\A\otimes\A$ and prove that $X$ itself is contained in \Proc. The integral is defined taking Riemann-type approximations, so it is straightforward to extend it to processes with finite variation. Using Remark~\ref{Rmk:BridgeMotion} we can therefore define the \textit{second component of the signature} of a free Brownian bridge $\beta$ on $[0,2\pi]$ by
\begin{align*}
	Z(t) & = \int_0^t\beta\tensor\d\,\beta\quad\quad t\in [0,2\pi]
\end{align*}

\noindent where the integral is in the sense of \cite{CDM_Levy}, see also \ts{Victoir}\cite{Victoir04}. 
\par If \A\ is a von Neumann algebra and $\phi$ a faithful tracial state on \A\ then its tensor product $\phi\tensor\phi$ is a faithful tracial state on the von Neumann tensor product $\A\tensor\A$ of $\A$ with itself, see for example \cite{Victoir04}, p. 109. So we can consider $(\A\tensor\A,\phi\tensor\phi)$ as a \NCP\ in its own right. We will discuss here the law of $Z(2\pi)$ with respect to this space.
\par We will also use the notation $\wh\A,\ \hphi$ for $\A\tensor\A$, $\phi\tensor\phi$ respectively.

\subsection{Using the L\' evy Representation}

The representation (\ref{Eq:Levy}) and a straightforward calculation using orthogonality of the trigonometric functions yield

\begin{prop}
	\label{Thm:LABB}
	The L\' evy area of the free Brownian bridge at time $2\pi$ is the random variable
	\begin{align}
		\label{Eq:BridgeArea}
 			Z(2\pi) &= \sum_{n=1}^\infty \inv{n}\left(\xi_n\otimes\eta_n - \eta_n\otimes\xi_n\right).
	\end{align}
\end{prop}

\noindent In order to further analyse this series we need to know how the $\xi_m\otimes\eta_m$, $\eta_m\otimes\xi_m$ are correlated. The following technical lemma is in a slightly more general framework than we need here.

\begin{lem}
	\label{Thm:TensorFree}
	Let $\{a_n,b_n\colon n\in\N\}$ be a collection of free random variables in \A\ such that for each $n$ the variables $a_n,b_n$ are identically distributed. Then the following set is free in $\A\otimes\A$:
	\begin{align*}
		\{a_n\otimes b_n, b_m\otimes a_m\colon m,n\in\N\}.
	\end{align*}
	\begin{proof}
		Let $X_k\in A(\unit, \alpha_{j_k}\tensor\beta_{j_k})$ such that $j_k\ne j_{k+1}$ for $k\in\{1,\ldots,m-1\}$ and for each $k$ we have $\{\alpha_k,\beta_k\}=\{a_{j_k},b_{j_k}\}$. Suppose moreover that each $\hphi(X_k)=0$. We need to show that $\hphi(X_1\ldots X_m)=0$. Note that $X_k=p_k(\alpha_k\tensor\beta_k)$ for some polynomial $p_k$. Since addition and multiplication in the tensor product act componentwise
	\begin{align*}
		\hphi(X_1\ldots X_m) &= \hphi\left((p_1(\alpha_{j_1})\tensor p_1(\beta_{j_1}))\ldots (p_m(\alpha_{j_m})\tensor p_m(\beta_{j_m}))\right)\\
	& = \hphi\left[(p_1(\alpha_{j_1})\ldots p_m(\alpha_{j_m}))\tensor (p_1(\beta_{j_1})\ldots p_m(\beta_{j_m})) \right]\\
	& = \phi\left(p_1(\alpha_{j_1})\ldots p_m(\alpha_{j_m})\right)\ \phi\left(p_1(\beta_{j_1})\ldots p_m(\beta_{j_m})\right).
	\end{align*}
	By freeness of the $a_n,b_n$ and the fact that $j_k\ne j_{k+1}$ each of $\alpha_{j_1},\ldots,\alpha_{j_m}$ and $\beta_{j_1},\ldots \beta_{j_m}$ are free. Since $\hphi(X_k)=0$ we have $\phi(p_k(\alpha_{j_k}))\phi(p_k(\beta_{j_k}))=0$. So one of the factors must vanish. But since $\alpha_{j_k},\beta_{j_k}$ are identically distributed, either both or none of them are zero. So $\phi(p_k(\alpha_{j_k}))=\phi(p_k(\beta_{j_k}))=0$ for all $k$. Freeness now implies that the last line, and hence $\hphi(X_1\ldots X_m)=0$, vanishes.
	\end{proof}
\end{lem}

\noindent So the set $\{\xi_n\otimes\eta_n,\eta_n\otimes\xi_n\colon n\in\N\}$, and hence the terms of the right hand side of (\ref{Eq:BridgeArea}), are free. Since the R-transform is additive on free random variables we will use this tool to compute the distribution of $Z(2\pi)$ in $(\wh\A,\hphi)$. From Lemma~\ref{Thm:TensorFree} we can deduce

\begin{cor}
The R-transform of $Z(2\pi)$ is given by
\begin{align}
	\label{Eq:RSum}
		R_{Z(2\pi)} (z) & = 2\,\sum_{n=1}^\infty \frac{1}{n}\, R_{\xi\tensor\eta}\left(\frac{z}{n}\right).
\end{align}
\end{cor}

\begin{rmk}
	By the definition of $\hphi$ we have $\hphi\left((\xi\tensor\eta)^k\right)=\phi(\xi^k)^2$ for $k\in\N$.
Recall that $R_a(z)=\sum_{m=1}^\infty k_{m+1}(a) z^m$ where $k_m(a)$ denotes the $m^{\text{th}}$ cumulant of $a$. In particular $k_1(\xi\tensor\eta)=\phi(\xi)^2=0$ so that (on a neighbourhood of zero) $R_{\xi\tensor\eta}(z)=zP(z)$ for some $P\in\C[z]$. Rewriting (\ref{Eq:RSum}) yields
\begin{align}
	R_{Z(2\pi)}(z)& = 2\, \sum_{n=1}^\infty \inv{n^2}\, P\left(\frac{z}{n}\right),
\end{align}
in particular the right hand side of (\ref{Eq:RSum}) converges in a neighbourhood of zero.
\end{rmk}

\subsection{The Distribution of $\xi\tensor\eta$ and Meanders}

\noindent We proceed to compute the R-transform of $\zeta:=\xi\tensor\eta$ with $\xi,\eta$ free standard semicirculars. Recall that the odd moments of $\xi$ vanish and that $\phi(\xi^{2n})$ is given by the $n^\text{th}$ \textit{Catalan number}
\begin{align}
	\phi(\xi^{2n}) &= C_n : = \inv{2n+1}\begin{pmatrix} 2n\\ n\end{pmatrix}.
	\intertext{Since $\xi$, $\eta$ are self-adjoint, so is $\zeta$. Hence its law is a probability measure $\nu$ with compact support in \R. In particular $\nu$ is determined by its moments which are given by}
	\label{MomentsMu}
	\int t^m\nu(\d t) = \phi\left((\xi\tensor\eta)^m\right) &= \phi(\xi^m)\phi(\eta^m) = \begin{cases}
				\left(C_k\right)^2\quad & \text{if } m=2k\\ 0 &\text{ if } m \text{ is odd} \end{cases}
\end{align}
 
\noindent i.e. $\nu$ is the law of $\zeta_1\zeta_2$ where the $\zeta_i$ are independent commutative random variables with standard semicircular distribution. Therefore $\nu$ is absolutely continuous with respect to Lebesgue measure with density $\phi$ given by	
\begin{align}
		\label{DensityMu}
		\phi(t)& = \inv{4\pi^2}\int_{-\abs{t}/2}^{\abs{t}/2} \sqrt{4-x^2}\, \sqrt{4x^2-t^2}\, \d x \ \ 1_{[-2,2]}(t).
	\end{align}

\par\noindent The Catalan numbers $C_n$ are well-known in combinatorics. They give, for example, the number of Dyck paths of length $2n$. Similarly there is a combinatorial interpretation of the squares of the Catalan numbers, as detailed in \ts{Lando}--\ts{Zwonkin}\cite{LandoZwonkin93} and \ts{Di Francesco}--\ts{Golinelli}--\ts{Guitter}\cite{DGG_Meander}: consider an infinite line in the plane and call it the \textit{river}. A \textit{meander} of order $n$ is a closed self-avoiding connected loop intersecting the line through $2n$ points (the \textit{bridges}). Two meanders are said to be \textit{equivalent} if they can be deformed into each other by a smooth transformation without changing the order of the bridges. If a meander of order $n$ consists of $k$ closed connected non-intersecting (but possibly interlocking) loops it is said to have $k$ \textit{components}.\\

\begin{figure}[H]
	\begin{center}
	\includegraphics[scale = .19]{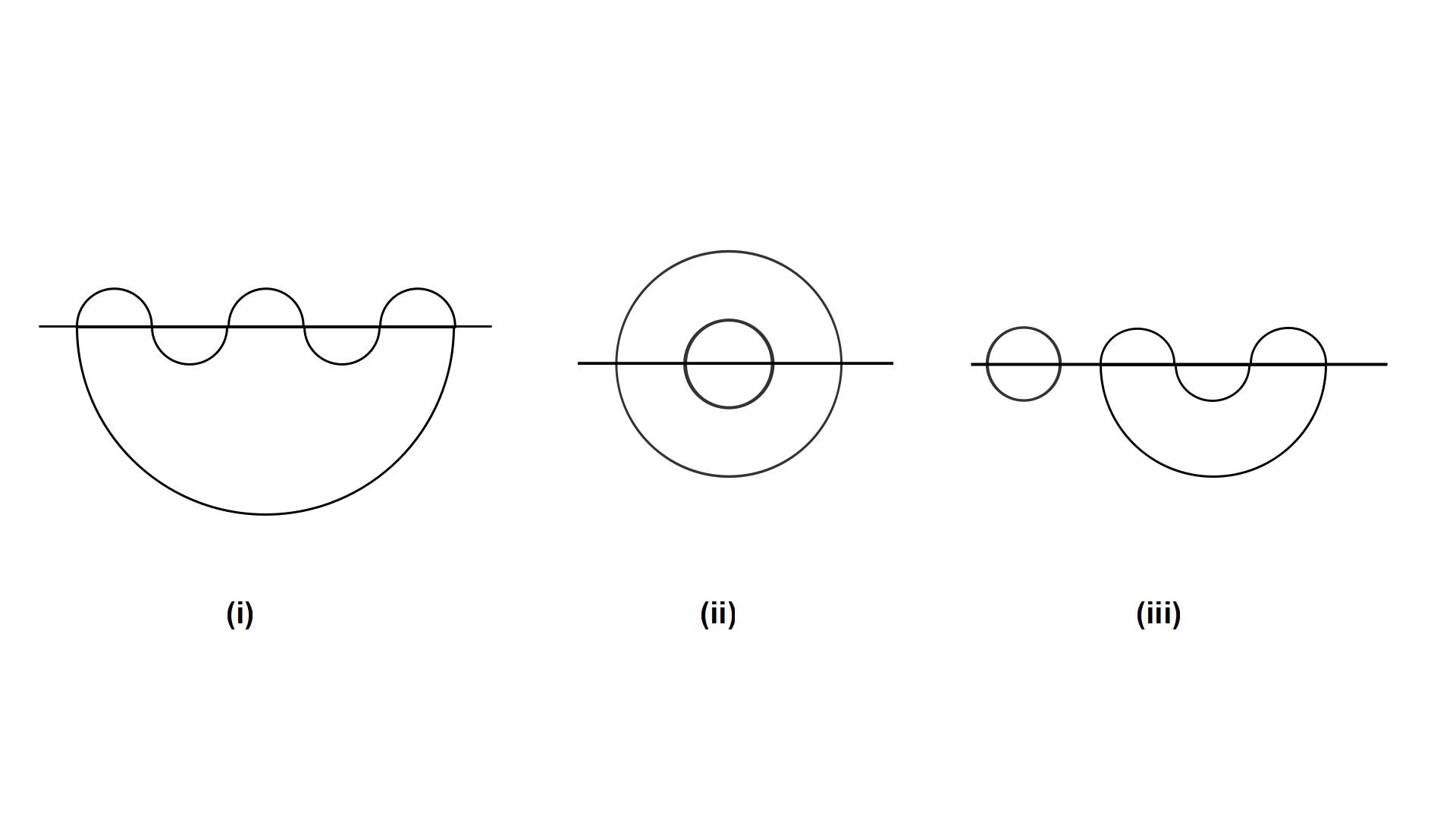}
	\parbox{10cm}{\caption{\textbf{(i)} 1-component meander of order 3; \textbf{(ii)} order 2, 2 components; \textbf{(ii)} order 3, 2 components}}
\end{center}

\end{figure}

\par\noindent A multi-component meander is said to be \textit{$k$-reducible} if a proper non-trivial collection of its connected components can be detached from the meander by cutting the river $k$ times between the bridges. Otherwise the meander is said to be \textit{$k$-irreducible}.
\begin{figure}[H]
	\begin{center}
	\includegraphics[scale = .19]{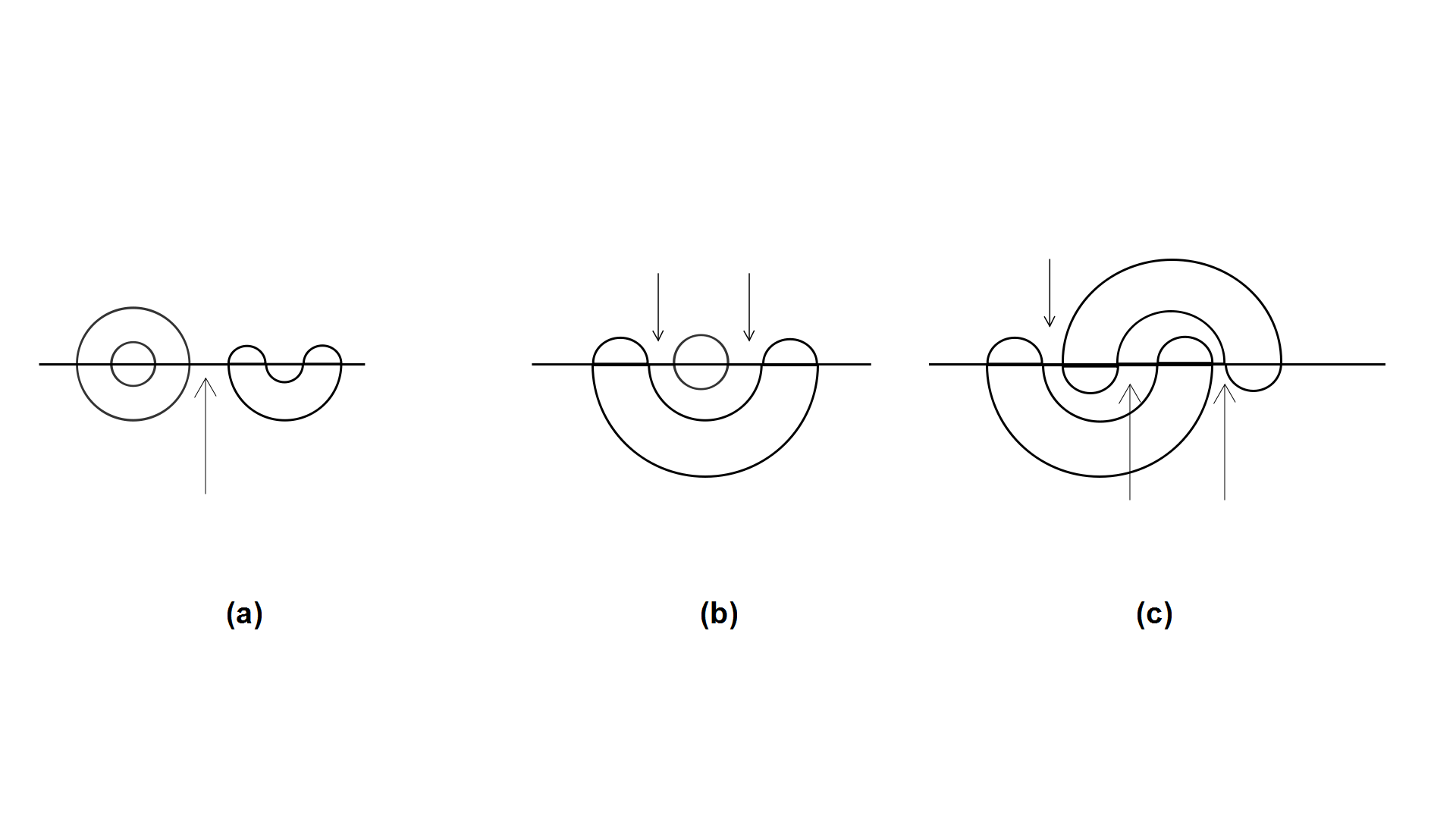}
	\parbox{11cm}{\caption{meanders that are \textbf{(a)} 1-reducible but 2-irreducible; \textbf{(b)} 1- and 2-reducible but 3-irreducible \textbf{(c)} 3-reducible}}
\end{center}

\end{figure}

\par\noindent The 2-irreducible meanders have been studied extensively in \cite{LandoZwonkin93} (where they are called irreducible meanders). Denote the generating series of the $q_m$ by $\Q$. Our connection to these objects is the following

\begin{prop}
	Let $q_n$ denote the number of 2-irreducible meanders of order $2n$ and $k_n=k_n(\xi\otimes\eta)$ the $n^\text{th}$ cumulant of $\xi\tensor\eta$. Then
	\begin{align}
		k_n(\xi\otimes\eta) &= \begin{cases}
					q_m  \quad&\text{if } n=2m\\
					0 &\text{if } n\text{ is odd}
				\end{cases}
	\end{align}
	\begin{proof}
		We first prove by induction that $k_n=0$ if $n$ is odd, which will follow from the fact that $\hphi((\xi\tensor\eta)^n)=0$ for $n$ odd. Assume that $k_m=0$ whenever $m<n$ is odd.  From (\ref{Eq:DefCum}) it follows that
		\begin{align}
			\notag			
			k_n & = \sum_{\substack{ \pi\in\NC(n)\\ \pi\ne \unit}} k_\pi
		\intertext{where $k_\pi=k_{V_1}\ldots k_{V_r}$ if $V_1,\ldots,V_r$ are the equivalence classes of $\pi$ and $\unit$ denotes the identity partition, i.e. $[k]_\unit=\ul{n}$. Every $\pi\in\NC(n)\setminus\{\unit\}$ must contain at least one equivalence class of size $m$ for some odd integer $m<n$. Since $k_m$ is a factor of $k_\pi$ and $k_m=0$, the inductive hypothesis implies $k_n=0$ as required. Hence}
	\notag
	R_{\xi\tensor\eta}(z) & = \sum_{n=1}^\infty k_{2n} z^{2n-1}.
	\intertext{Define the \textit{moment series} of $\xi\tensor\eta$ by}
	\notag	
	M(z)& = \inv{z}G\left(\inv{z}\right) = 1+\sum_{n=1}^\infty \hphi\left((\xi\tensor\eta)^n\right)\, z^n.
	\intertext{It is a consequence of the relationship between Cauchy and R-transform that}
	\label{DetR}
	M(z)&=1+zM(z)\, R(zM(z)).
	\intertext{We will introduce one more generating series. Put}
	\notag
	\rho(z)& = \sum_{n=1}^\infty q_n z^{2n-1}
	\intertext{so that $\Q(z) = 1+z\rho(z)$. From (7.10) in \cite{DGG_Meander} we have}
	\label{MRho}
	M(z) & = q(z M(z)) = 1+z M(z)\, \rho(zM(z)).
\end{align}
	Combining (\ref{DetR}) and (\ref{MRho}) yields $\rho=R$ as power series. That $k_{2n}=q_n$ now follows from comparing coefficients.
	\end{proof}
\end{prop}

\subsection{The Distribution of $Z(2\pi)$}

\par\noindent So we have an explicit expression for the R-transform of $\xi\tensor\eta$. We will use this to obtain the R-transform of $Z(2\pi)$. 
\par Recall that all odd cumulants of $\xi_n\otimes\eta_n$ and $\eta_n\otimes\xi_n$ vanish, hence the same is true of $Z(2\pi)$.

\begin{prop}
	\label{Thm:CumZ}
	The $2n^\text{th}$ cumulant of $Z(2\pi)$ is $2\, \zeta(2n)\,q_n$ where $\zeta$ is the Riemann zeta function.
\begin{proof}
	Recall that $\zeta(m)=\sum_{n=1}^\infty n^{-m}$. So
	\begin{align*}
		R_{Z(2\pi)}(z) &= 2\sum_{n=1}^\infty \inv{n}\, R_{\xi\tensor\eta}\left(\frac{z}{n}\right)
			  = 2\sum_{n=1}^\infty \inv{n} \sum_{m=1}^\infty k_m\left(\frac{z}{n}\right)^{m-1} \\ &= 2\sum_{n=1}^\infty \sum_{m=1}^\infty n^{-2m} q_m z^{2m-1} \\
& = \sum_{m=1}^\infty 2\,\zeta(2m) q_m z^{2m-1}
	\end{align*}
	where interchanging the sums over $m$ and $n$ is justified by absolute convergence.
\end{proof}
\end{prop}

\begin{defn}[see \cite{Stanley1}, p. 107]
	Let $(a_n)_{n\in\N}$, $(b_n)_{n\in\N}$ be two sequences with generating functions $f,g$ respectively. The \textit{Hadamard product} of $f, g$ is defined to be the generating function of $(a_nb_n)$, denoted $f\boxast g$. That is
\begin{align*}
	f\boxast g(z)& = \sum_{n=1}^\infty a_n b_n z^n.
\end{align*}
\end{defn}

\par\noindent So $R_{Z(2\pi)}$ is twice the Haddamard product of the generating functions of the 2-irreducible meanders and that of the sequence $\{\zeta(2m)\colon m\in\N\}$.
\par From (6.3.14) in \ts{Abramowitz--Stegun} \cite{HandMathFn} we have for $\abs{z}<1$,
	\begin{align*}
	\label{Eq:ZetaPsi}
\sum_{n=2}^\infty \zeta(n+1) z^{n} &  = -\gamma-\Psi(1-z)
\intertext{where $\gamma$ is the Euler constant and $\Psi$ is the \textit{Digamma function} defined by}
	\notag
		\Psi(z) & = \frac{\d}{\d z}\log\Gamma(z) = \frac{\Gamma'(z)}{\Gamma(z)}.
	\end{align*}

\noindent Since the generating series can be considered as functions inside their radius of convergence, we can use complex analysis to compute their Hadamard product. Namely

\begin{lem}
	\label{Thm:Had}
	 Let $f,g$ be generating functions of  $(a_n)_{n\in\N}$, $(b_n)_{n\in\N}$ and suppose that they are analytic on a neighbourhood of 0. Then 
	\begin{align}
	(f\boxast g)(z^2) & = \inv{2\pi i}\int_\gamma f(zw)\, g\left(\frac{z}{w}\right)\, \frac{\d w}{w}
\end{align}
\noindent on a neighbourhood $U$ of 0, where $\gamma$ is a smooth closed curve around 0 and contained in $U$.
\begin{proof}
	Let $U_1,U_2$ be neighbourhoods of 0 on which $f$ and $g$ respectively are analytic. Then for $z\in U=U_1\cap U_2$,
	\begin{align*}
		\inv{2\pi i} \int_\gamma f(zw)\, g\left(\frac{z}{w}\right)\, \frac{\d w}{w} & = \left[f(z\eta)\, g\left(\frac{z}{\eta}\right)\right]_{\eta^0}\\
	& = \left[\sum_{n=0}^\infty a_n (z\eta)^n\sum_{m=0}^\infty b_m \left(\frac{z}{w}\right)^m\right]_{\eta^0}\\
	& = \left[\sum_{m,n} a_n b_m z^{n+m}\eta^{n-m}\right]_{\eta^0}\\
	& = \sum_{n=0}^\infty a_n b_n z^{2n} = f\boxast g(z^2)
	\end{align*}
	\noindent where $[\cdot]_{\eta^0}$ denotes the constant term in a Laurent series in $\eta$.
\end{proof}	
\end{lem}

\begin{cor}
	 Let $\epsilon\in(0,\rho)$ where $\rho$ is the radius of convergence of $R_{Z(2\pi)}$ and choose the canonical branch of the square root on $B(0,\rho)$. Then for $z\in B(0,\rho)$
\begin{align}
	R_{Z(2\pi)}(z)&= -\frac{z^{1/2}}{\pi i}\int_\Gamma \Psi(1-z^{1/2}w)\,\Q\left(\frac{z}{w}\right)\,\d w
\end{align}
\noindent where $\Gamma=\partial B(0,\rho)$.
\begin{proof}
	By Proposition~\ref{Thm:CumZ} we have $R_{Z(2\pi)} = 2\, \Q\boxast\Lambda$ where, using (\ref{Eq:ZetaPsi})
	\begin{align*}
	\Lambda(z)&=\sum_{n=1}^\infty \zeta(m)z^m=-z\Psi(1-z)-\gamma z.
	\intertext{Lemma~\ref{Thm:Had} now yields}
	 (\Q\boxast\Lambda)(z^2)&= \inv{2\pi i}\int_\Gamma \Lambda(z w)\,\Q\left(\frac{z}{w}\right)\, \frac{\d w}{w}\\
			& =-\inv{2\pi i}\int_\Gamma zw\left(\Psi(1-zw) +\gamma\right)\,\Q\left(\frac{z}{w}\right)\, \frac{\d w}{w}\\
			& = -\inv{2\pi i}\int_\Gamma z\Psi(1-zw)\,\Q \left(\frac{z}{w}\right)\,\d w\\
		& \quad-\frac{\gamma\, z}{2\pi i}\int_\Gamma \Q\left(\frac{z}{w}\right)\,\d w.
	\intertext{The argument of the integral in the second summand has a power series with only even powers of $w$ so the integral itself must vanish. We therefore have}	
	(q\boxast\Phi)(z^2) & = \frac{z}{2\pi i}\int_\Gamma \Psi(1-zw)\,q\left(\frac{z}{w}\right)\,\d w
\end{align*}
\end{proof}
\end{cor}

\begin{rmk}
	In\cite{DGG_Meander} it has been shown that the radius of convergence of $\Q$ is $\frac{4}{\pi}-1$. Since $\zeta(m)\longrightarrow 1$ as $m\longrightarrow\infty$, it follows that the radius of convergence of $R_{Z(2\pi)}$ is also $\frac{4}{\pi}-1$.
\end{rmk}

\par\noindent It is well-known, see \cite{HiaiPetz}, that the semicircular distribution is \FID. By Lemma~\ref{Thm:TensorFree} it follows that the $\xi_n\otimes\eta_n$ and $\eta_n\otimes\xi_n$ are \FID. Since free infinite divisibility is preserved by free linear combinations and weak limits, it follows that $Z(2\pi)$ is also \FID.
\par Unfortunately it seems that there is no explicit formula for \Q. It is therefore not apparent how a similar analysis to that for the square norm could be applied in order to obtain further details about the distribution of $Z(2\pi)$.


\section{L\'evy Area of the Free Brownian Bridge}

\label{Sec:FLA}

In this section we use the L\'evy representation

\begin{align}
	\label{Eq:LevyRep}
	\beta(t)  = \sum_{n=1}^\infty \frac{\cos(nt)-1}{n\sqrt\pi}\,\xi_n &+ \sum_{n=1}^\infty\frac{\sin(nt)}{n\sqrt\pi}\,\eta_n
\intertext{of the free Brownian bridge to compute the distribution of the free analogue of the classical L\'evy area process defined by}
	\label{Eq:DefCLA}
	\l(t) = \frac{i}{2} \int_0^t \left[\beta(s),\d\beta(s)\right] & =\frac{i}{2}\int_0^t\left(\beta(s)\d\beta(s) - \d\beta(s)\beta(s)\right).
	\intertext{When $\beta$ is a two-dimensional commutative Brownian motion this is very similar to the object studied by \ts{L\'evy}\cite{Levy50}. By standard properties of the non-commutative integral \cite{BianeSpeicher98} and self-adjointness of $\beta$ we have}
	\notag
		\int_0^t\beta(s)\d\beta(s) & = \left(\int_0^t\d\beta(s)\beta(s)\right)^*.
	\intertext{A straightforward calculation yields that the left hand side equals, for $t=2\pi$,}
	\label{Eq:CLASum}
	\int_0^{2\pi}\beta(s)\d\beta(s) & = \sum_{n=1}^\infty \inv{n}\left(\xi_n\eta_n-\eta_n\xi_n\right)
\end{align}
which is easily seen to be anti-self-adjoint. This is the reason for the factor of $i$ in (\ref{Eq:DefCLA}): multiplying an anti-self-adjoint operator by $i$ yields a self-adjoint random variable whose distribution is therefore supported in \R. Thus $\l:=\l(2\pi)$ is equal to either side of (\ref{Eq:CLASum}) multiplied by $i$. 
\par The summands are \textit{commutators} of free semicircular random variables. Commutators have been studied by \ts{Nica--Speicher}\cite{NicaSpeicher98}, where the semicircle distribution is discussed in Example~1.5(2). If $c_n=i\left(\xi_n\eta_n-\eta_n\xi_n\right)$, then the support of $\mu_{c_n}$ is $[-r,r]$ where $r=\sqrt{\frac{11+5\sqrt{5}}{2}}$ and
\begin{align}
	\label{Eq:RCn}
	R_{c_n}(z) &= \frac{2z}{1-z^2} = 2\sum_{m=1}^\infty z^{2m-1}.
\end{align}
From this we can now compute the R-transform of the classical L\'evy area. Let that function be denoted $R_{\l}$ then
\begin{align}
	\notag
	R_{\l} & = \sum_{n=1}^\infty \inv{n} R_{c_n}\left(\frac{z}{n}\right) = \sum_{n=1}^\infty \frac{2n}{n^2-z^2}\\
	\label{Eq:RLevyArea}
	 & = \inv{z} - \pi\cot(\pi z).
\end{align}
We can deduce the free cumulants of $\l$, either from the Taylor series of (\ref{Eq:RLevyArea}) or by calculating
\begin{align}
	\notag
	R_{\l} & = \sum_{n=1}^\infty \frac{2}{n} \sum_{m=1}^\infty \left(\frac{z}{n}\right)^{2m-1} = \sum_{m=1}^\infty 2 \left(\sum_{n=1}^\infty n^{-2m}\right)\,z^{2m-1}\\
	\notag
	& = \sum_{m=1} 2\zeta(2m) z^{2m-1}
\end{align}
where the interchanging of the infinite sums is justified by absolute convergence. The free cumulants of \l\ are therefore given by
\begin{align}
k_m(\l) & = \begin{cases} 2\zeta(m)\quad&\text{if } m\text{ is even}\\ 0& \text{otherwise.}\end{cases}
\end{align}

\par\noindent Free infinite divisibility is characterised by an analytic property of the R-transform. An analytic function $f\colon\C^+\map\C^+$ is called a \textit{Pick function}. For $a,b\in\R$ with $a<b$ we denote by $\p(a,b)$ the set of Pick functions $f$ which have an analytic continuation $g\colon\C\setminus\R\cup (a,b)\map\C$ such that $g(\overline{z})=\overline{g(z)}$. The following result is Theorem~3.3.6 of \textsc{Hiai--Petz}\cite{HiaiPetz}:

\begin{thm}
\label{Thm:AnalyticFID}
A compactly supported probability measure $\mu$ is \FID\ if and only if its R-transform extends to a Pick function in $\p(-\epsilon,\epsilon)$ for some $\epsilon>0$.
\end{thm}

\noindent It is easy to see that the common R-transform of the $c_n$ extends to a Pick function in $\p(-1,1)$. Therefore each $c_n$ is \FID.

\begin{cor}
	The distribution of $\l$ is \FID.
\end{cor}

\par\noindent As in Section~\ref{Sec:SqNorm} we can use free infinite divisibility together with the analytic properties of the R-transform and the  formula for the maximum of the support from \cite{LDP_NCP} to describe further the distribution in question. 

The variational formula of Section~\ref{Sec:SqNorm} (Theorem~\ref{Thm:EdgePositive}) assumed that all free cumulants are positive, which is not the case for $\l$ (which is symmetric and therefore has vanishing odd free cumulants). However non-negativity of all free cumulants is actually enough \cite[Theorem 5.9]{LDP_NCP}:

\begin{thm}
	\label{Thm:EdgeNonnegative}
	Let $a\in\A$ be a self-adjoint non-commutative random variable with distribution $\mu$ and free cumulants $k_m\geq 0$ for all $m$. Then the right edge $\rho_\mu$ of the support of $\mu$ is given by
	\begin{align}
		\label{Eq:EdgeNonnegative}
		\log\left(\rho_\mu \right) & = \sup\left\{\inv{m_1(p)}\sum_{n\in L} p_n\log\left(\frac{k_n}{p_n}\right) - \frac{\Theta(m_1(p))}{m_1(p)} \colon p\in\m_1^1(L)\right\}
	\end{align}	
	where $\m_1^1(L)$ denotes the set of $p\in\m_1^1(\N)$ such that $p(L^c)=0$ and $\Theta$ was defined in Theorem~\ref{Thm:EdgePositive}.
\end{thm}

\par\noindent The inverse of the Cauchy transform of $\l$ is given by
\begin{align*}
	K_{\l} & = R_{\l}+\inv{z} = \frac{2}{z} - \pi\cot(\pi z).
\end{align*}
We can check, by simple if lengthy computations similar to those in Section~\ref{Sec:FIDBP} that for every $t\in (\pi,2\pi)$ there exists unique $r(t)>0$ such that $\Im \left[ K_\l \left(r(t) e^{it}\right)\right]=0$ and that
\begin{align}
	\left. \frac{\partial}{\partial z} \Im \left[ K_\l(z) \right] \right\vert_{z=r(t) e^{it}}  &\ne 0 \quad \quad\forall t\in (\pi,2\pi).
\end{align}
We obtain the following characterisation of the distribution of \l:

\begin{prop}
	\label{Thm:DistFLA}
	The non-commutative random variable $\l$ is distributed according to $\mu_\l(\d t) = \Phi_\l( t)\one_{[-\rho_\l,\rho_\l]}\, \d t$ where $\Phi_\l(x) = -\inv{\pi}r(t_x)\sin(\tau_x)$ and $\tau_x$ is the unique solution on $(\pi,2\pi)$ to
\begin{align}
	\label{Eq:TauXLA}
	\frac{2}{r(\tau_x)\,e^{i\tau_x}} & - \pi \cot\left(\pi r(\tau_x)\,e^{i\tau_x}\right)=x.
\intertext{for every $x\in (-\rho_\l,\rho_\l)$. The number $\rho_\l$ is given by}
		\label{Eq:RhoLA}
		\rho_\l & = \frac{m_*\pi}{\sqrt{m_*^2-2}}
		\intertext{where $m_*$ is the unique solution on $(\sqrt{2},\infty)$ of}
		\label{Eq:MLA}
		m-2 & = \sqrt{m^2-2}\, \cot\left(\frac{\sqrt{m^2-2}}{m-1}\right).
	\end{align}
\end{prop}

\begin{figure}[H]
	\begin{center}
	\includegraphics[scale = .3]{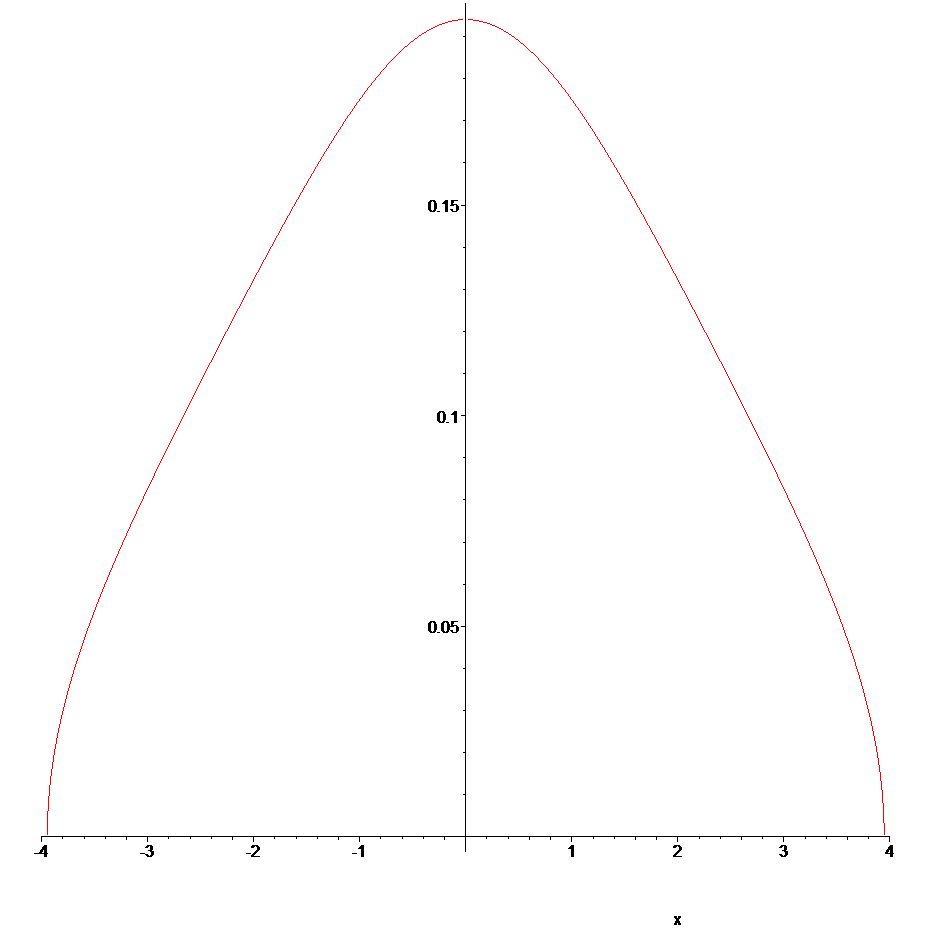}
	\hspace{5cm}
	\parbox{11cm}{\caption*{\footnotesize{Density of the free L\'evy area}}}
\end{center}

\end{figure}

\begin{proof}[Proof of Proposition~\ref{Thm:DistFLA}]
The law $\mu_\l$ of \l\ is symmetric about 0. Together with the analytic arguments of Section~\ref{Sec:FIDBP}, suitably modified, this implies the existence of $\rho_\l>0$ such that the density $\Phi_\l$ of $\mu_\l$ is smooth, positive on $(-\rho_\l,\rho_\l)$ and zero everywhere else. The function $\Phi_\l$ is given by $\Phi_\l(x) = -\inv{\pi}r(\tau_x)\sin(\tau_x)$ where $\tau_x$ is characterised by (\ref{Eq:TauXLA}).

For the remainder of the statement we apply Theorem~\ref{Thm:EdgeNonnegative}. Only the free cumulants of even order are nonzero, so that $L=\{2n\colon n\in\N\}$. The supremum on the right-hand side of (\ref{Eq:EdgeNonnegative}) is attained by a unique maximiser which gives rise to equations (\ref{Eq:RhoLA}) and (\ref{Eq:MLA}). This completes the proof of the proposition.
\end{proof}

\ \\

\ \\\ \\ \noindent \textsc{Mathematics Institute, University of Warwick, Coventry CV4 7AL, UK}\\\ \\ \textit{Email address} \texttt{j.ortmann@warwick.ac.uk}

\end{document}